\documentclass[12pt]{article}
\usepackage{latexsym}
\usepackage{amsmath}
\usepackage{amssymb}
\usepackage{amsthm}
\usepackage{mathtools}
\usepackage{dsfont}
\usepackage{bbm}
\usepackage{color}
\usepackage[tmargin=1.25in,bmargin=1.25in,lmargin=1.1in,rmargin=1.1in]{geometry}

\usepackage{graphicx}
\usepackage{epsfig}
\usepackage{hyperref} 

\definecolor{Red}{rgb}{1,0,0}
\definecolor{halfgray}{gray}{0.55}
\definecolor{webgreen}{rgb}{0,.5,0}
\definecolor{webbrown}{rgb}{.6,0,0}
\definecolor{Maroon}{cmyk}{0, 0.87, 0.68, 0.32}
\definecolor{RoyalBlue}{cmyk}{1, 0.50, 0, 0}
\definecolor{Black}{cmyk}{0, 0, 0, 0}

\hypersetup{
    colorlinks=true, linktocpage=true, pdfstartpage=1, pdfstartview=FitV,
    breaklinks=true, pdfpagemode=UseNone, pageanchor=true, pdfpagemode=UseOutlines,
    plainpages=false, bookmarksnumbered, bookmarksopen=true, bookmarksopenlevel=1,
    hypertexnames=true, pdfhighlight=/O,
    urlcolor=webbrown, linkcolor=RoyalBlue, citecolor=webgreen,
}

\newtheorem{thm}{Theorem}[section]
\newtheorem{proposition}[thm]{Proposition}

\newtheorem{lem}[thm]{Lemma}

\newtheorem{remark}[thm]{Remark}

\makeatletter\@addtoreset{equation}{section}\makeatother

\begin{document}

\newcommand{\D}{\mathrm{d}}
\newcommand{\R}{\mathbb{R}}
\newcommand{\E}{\mathbb{E}}
\newcommand{\var}{\operatorname{Var}}
\newcommand{\cov}{\operatorname{Cov}}
\newcommand{\G}{\mathcal{G}}
\newcommand{\B}{\mathcal{B}}
\newcommand{\e}{\mathrm{e}}
\newcommand{\p}{\mathrm{P}}
\newcommand{\N}{\mathbb{N}}
\newcommand{\K}{\mathbb{K}}
\newcommand{\T}{\mathbb{T}}
\newcommand{\bbP}{\mathbb{P}}
\newcommand{\bbE}{\mathbb{E}}
\newcommand{\bfP}{\mathbf{P}}
\newcommand{\bfE}{\mathbf{E}}
\newcommand{\calE}{\mathcal E}
\newcommand{\GW}{\operatorname{GW}}
\newcommand{\Tra}{\operatorname{Tra}}
\newcommand{\RW}{\operatorname{RW}}
\newcommand{\RWP}{\RW\times\bfP}
\newcommand{\F}{\mathcal F}

\pagenumbering{arabic}

\newcommand{\Peaug}{\mathrm{P}_\varepsilon^{\mathrm{aug}}}
\newcommand{\Pehat}{\hat{\mathrm{P}}_\varepsilon}
\newcommand{\PPehat}{\hat{\mathbb{P}}_\varepsilon}
\newcommand{\Eehat}{\hat{\mathrm{E}}_\varepsilon}
\newcommand{\EEehat}{\hat{\mathbb{E}}_\varepsilon}

\title{The speed of random walk on Galton-Watson trees with vanishing conductances}
\author{Tabea Glatzel, Jan Nagel} 
\maketitle

\begin{abstract} 
In this paper we consider random walks on Galton-Watson trees with random conductances. On these trees, the distance of the walker to the root satisfies a law of large numbers with limit the effective velocity, or speed of the walk. We study the regularity of the speed as a function of the distribution of conductances, in particular when the distribution of conductances converges to a non-elliptic limit. 
\end{abstract}


\smallskip


\section{Introduction}
\label{sec:intro}

Let $T$ be a supercritical Galton-Watson tree and let $(X_n)_{n\geq 0}$ be a simple random walk on $T$, starting at the root. It is well known since the work of \cite{LyoPemPer96a}, that $(X_n)_{n\geq 0}$ moves away from the root with a linear rate,  
\begin{align} \label{basicLLN}
\lim_{n\to \infty} \frac{|X_n|}{n} = v \quad \text{a.s.},
\end{align}
where $|v|$ denotes the distance of a vertex $v$ from the root and the limit holds almost surely with respect to the annealed law, that is, averaged over the random walk and the tree, and conditioned on $T$ being infinite. The limit $v$ is a positive constant and is called the effective velocity of the random walk. In fact, $v$ can be given explicitly in terms of the offspring law. Such a law of large numbers still holds for a random walk on a Galton-Watson tree with random conductances. For this process, given the Galton-Watson tree $T$, the edges of the tree are assigned random positive conductances, independently and identically distributed for different edges and the random walk traverses an edge with a probability proportional to the conductance of that edge. Is was proven by \cite{gantert2012random}, that a convergence as in \eqref{basicLLN} still holds, when we now also average over the conductances, with a limit depending on the offspring law of the tree and the law of the conductances. \cite{gantert2012random} also give a formula for the speed as an expectation of ratios of effective conductances. Unfortunately, this means that the speed cannot be explicitly computed even for very simple distributions of conductances. 

In this paper, we study the regularity of the speed as a function of the distribution of the conductances. In particular, we  investigate how the speed changes when the conductances on a positive fraction of edges goes to zero. The regularity of the speed as a function of the local transition probabilities is a prominent question for random walk on Galton-Watson trees. In particular, it has been well studied for biased random walks, when the conductances of edges between the $n$-th and the $(n+1)$-th generation in the tree are multiplied by $\lambda^n$, for a bias parameter $\lambda >0$. The speed then depends in a highly nontrivial way on $\lambda$ \cite{LyoPemPer96b}. For this model, the monotonicity of $v$ as a function of $\lambda$ has been studied by \cite{benarous2014monotonicity}, the behavior for $\lambda$ close to zero by \cite{benarous2013einstein} and differentiability by \cite{bowditch2019differentiability}.  
Closely related are results for random walks in $\mathbb{Z}^d$ with random conducances and bias parameter $\lambda$, as studied by \cite{fribergh2014phase,gantert2017einstein,berger2019speed}, or for the Mott random walk \cite{faggionato2018velocity,faggionato2019regularity}. The regularity of the speed  on the tree as a function of the offspring law was studied in \cite{HHN2020random}, when the offspring law is close to criticality.

Before we present our results, let us introduce the model more precisely. Let $\Omega$ be the set of all tripel $(T,\rho,\xi)$, where $T$ is a tree with root $\rho$ and $\mathcal{E}(T)$ its set of undirected edges and $\xi\in [0,\infty)^{\mathcal{E}(T)}$ is a configuration of conductances on the edges of $T$. An element $\omega\in \Omega$ is called an environment. 
We then let $\mathrm{P}$ be the law on $\Omega$, such that under $\mathrm{P}$, $T$ is a Galton-Watson tree with offspring law $\nu$ and root $\rho$ and conditioned on $T$, the conducances $\xi$ are independent and identically distributed with marginal law $\mu$. Let $N$ be the number of offspring of the root $\rho$. We assume that the tree is supercritical, $\mathrm{E}[N]>1$, it has no leaves, $\mathrm{P}(N=0)=0$ and the second moment $E[N^2]$ is finite. Additionally, the conductances are uniformly elliptic, that is, there exists a $\delta>0$, such that
\begin{align} \label{uniformelliptic}
\mu([\delta,\delta^{-1}])=1 .
\end{align}
For vertices $u,v\in T$, we write $u\sim v$, if $u$ and $v$ are connected by an edge in $T$ and we say that $u$ and $v$ are neighbors. We denote the graph distance of a vertex $v\in T$ to the root $\rho$ by $|v|$. 

Given an environment $\omega=(T,\rho,\xi)$, we let $P_\omega$ be the quenched law of the Markov chain $(X_n)_{n\geq 0}$ on $T$, starting at $\rho$ and with transition probabilities 
\begin{align} \label{quenchedlaw}
P_\omega (X_{n+1}=v | X_n=u) = \frac{\xi(u,v)}{\sum_{w\sim u} \xi(u,w)}
\end{align}
if $u\sim v$. The annealed law $\mathbb{P}$ is then defined by 
\begin{align} \label{annealedlaw}
\mathbb{P}(A\times B) = \int_A P_\omega(B) \mathrm{d} \mathrm{P}(\omega)
\end{align}
for measurable subsets $A\times B$ of the set of all pairs $(\omega,T^{\mathbb{N}_0})$ (as usual, the $\sigma$-algebras are generated by finite-dimensional projections). 

As proven by \cite{gantert2012random}, we then have 
\begin{align} \label{defspeed}
\lim_{n\to \infty} \frac{|X_n|}{n} = v =v(\nu,\mu) \quad \mathbb{P}-\text{almost surely.}
\end{align}
In Section \ref{sec:invariant}, we recall a formula for $v$. We are interested in the behavior of $v(\nu,\mu)$ as a function of $\mu$, the marginal law of the conductances. As a first result, we have that $v$ is continuous, as long as we stay in the framework of uniformly elliptic laws. More precisely, to guarantee the convergence of $v(\nu,\mu_n)$ for a sequence of distributions of conductances $\mu_n$, we require the $\mu_n$ to be uniformly elliptic with the same ellipticity constant $\delta$.

\begin{proposition} \label{prop:cont}
For any $\delta>0$, the mapping $\mu \mapsto v(\nu,\mu)$ is continuous on the set of uniformly elliptic measures satisfying \eqref{uniformelliptic}, equipped with the weak topoplogy. 
\end{proposition}

We want to investigate the speed when some of the conductances approach zero and define for $\mu$ satisfying \eqref{uniformelliptic} for some $\delta>0$ and for $\varepsilon \geq 0$, $\alpha\in (0,1)$ the measure
\begin{align} \label{defmuepsilon}
\mu_\varepsilon = \alpha \delta_\varepsilon + (1-\alpha) \mu .
\end{align}
Note that for any $\varepsilon>0$, the measure $\mu_\varepsilon$ is still uniformly elliptic. There is a natural way of coupling weighted trees $(T,\xi)$ with marginal conductance law $\mu$ with trees with conductances distributed according to $\mu_\varepsilon$. For each edge in $T$ we toss an independent coin with success probability $\alpha$ to replace the current conductance by $\varepsilon$. For $\varepsilon=0$, the subtree $T_0$ formed by edges with positive conductances might be finite, and then the random walk can only move on a finite tree. In fact, if for $\varepsilon=0$ we have $T_0$ finite with probability 1 (that is, the pruned Galton-Watson process obtained by removing $\varepsilon$-edges is not supercritical), we set $v(\nu,\mu_0)=0$. If $T_0$ is still supercritical, it is infinite with positive probability and then the speed is usually defined when we condition on $T_0$ to be infinite. In this case, and consistently with the definitions in \cite{LyoPemPer96a} and \cite{gantert2012random}, we define the conditioned environment law $\bar{\mathrm{P}} = \mathrm{P}(\cdot |\, |T_0|=\infty )$ and the conditioned annealed law $\bar{\mathbb{P}}$ as in \eqref{annealedlaw}, with $\mathrm{P}$ replaced by $\bar{\mathrm{P}}$. Then $v(\nu,\mu_0)$ is the almost sure limit of $|X_n|/n$ under $\bar{\mathbb{P}}$, that is, the (traversable) tree is conditioned to be infinite.      
The following theorem gives the limit of the speed as $\varepsilon$ tends to zero.

\begin{thm} \label{thm:main}
It holds that 
\begin{align}
\lim_{\varepsilon \searrow 0} v(\nu,\mu_\varepsilon) = \hat{\mathrm{P}}_0(|T_0|=\infty) \cdot v(\nu,\mu_0) , 
\end{align}
where $\hat{\mathrm{P}}_0$ is the invariant measure for the tree seen from the random walk (see Section 2).  
\end{thm}

Theorem \ref{thm:main} shows that the limit of the speed for conductances approaching zero is not the speed of the random walk as usually defined on trees with positive extinction probability. Instead it is smaller by the multiplicative constant $\hat{\mathrm{P}}_0(|T_0|=\infty)<1$. The reason for this is a slowdown experienced by the walk as it spends time in finite parts of the tree which can only be left by traversing an edge with conductance $\varepsilon$, acting as traps in the environment. Indeed, for small $\varepsilon$, the tree can be thought of as a forest formed by the edges with conductance larger than $\varepsilon$, connected by $\varepsilon$-edges. The walker can only enter or leave the finite trees formed by large conductances by moving along a $\varepsilon$-edge, which has a probability of order $\varepsilon$. Consequently, the time spend on finite trees connected by large conductances is of order $\varepsilon^{-1}$, resulting in a macroscopic amount of time spend in traps. For $\varepsilon=0$ and conditioned on $T_0$ being infinite, as in the definition of $v(\nu,\mu_0)$, such a slowdown cannot occur. It is interesting to note that when $T_0$ is conditioned to be infinite, the random walk moves on a tree with leaves, which create a different slowdown mechanism by forming traps. The even smaller limit of $v(\nu,\mu_\varepsilon)$ shows that the latter effect is in some sense weaker.

The reason the slowdown is not just given by the probability for an infinite tree under $\mathrm{P}$, is that in order to move with linear speed, we need to see an infinite tree from the perspective of the walk after a long time. We therefore have to look for an infinite tree under the measure $\hat{\mathrm{P}}_0$, which generates a tree stationary for the process of the environment viewed from the random walk. 

In the next section, we introduce the invariant measure for the environment seen from the walker and its limit as $\varepsilon\to 0$. The proofs of Proposition \ref{prop:cont} and Theorem \ref{thm:main} are given in Section \ref{sec:proofs}.

\section{The invariant measure}
\label{sec:invariant}

In order to define the invariant measure, we need to introduce augmented Galton-Watson trees and bi-infinite walks on them. We follow \cite{gantert2012random}, with modified notation. We write $\mathrm{P}_\varepsilon$ for $\mathrm{P}$, if the marginal law of the conductances is given by $\mu_\varepsilon$ as in \eqref{defmuepsilon}. Then, let $\Peaug$ be the law on $\Omega$, such that under $\Peaug$, $T$ is an augmented Galton-Watson tree with an additional edge $(\rho,v_0)$ added to the root, and attached to vertex $v_0$ at the other end of that edge an independent Galton-Watson tree. That is, 
\begin{align} \label{defaugmented}
\Peaug((T,\rho,\xi)\in A) = \sum_{k=1}^\infty p_{k-1}\mathrm{P}_\varepsilon((T,\rho,\xi)\in A\mid \mathrm{deg}(\rho) = k) .
\end{align} 
The measure $\Pehat$ is then defined via
\begin{align} \label{defPhat}
\Eehat [f(T,\rho, \xi)] = \mathrm{E}_\varepsilon^{\mathrm{aug}}\left[ f(T,\rho, \xi) \frac{\pi(\rho)}{\gamma \mathrm{deg}(\rho)}\right],
\end{align}
with $\pi(v) = \sum_{w\sim v} \xi(v,w)$ and $\gamma$ the normalization constant. Under $\Pehat$, the root is weighted by the average conductance of adjacent edges and $\gamma$ is the mean conductance of an edge. Given an environment $\omega\in \Omega$, let $\hat P_\omega$ be the law of the bi-infinite random walk $(X_n)_{n\in \mathbb{Z}}$ on $T$, such that $X_0=\rho$ and $(X_n)_{n\geq 0}$ (interpreted as the future of the random walk) and $(X_{-n})_{n\geq 0}$ (interpreted as the past of the random walk) are independent with marginal law as defined by \eqref{quenchedlaw}. We let $\PPehat$ be the corresponding annealed law, defined analogously to \eqref{annealedlaw}. As verified in \cite[Equation (2.2)]{gantert2012random}, we have for any bounded and measureable functions $f,g$
\begin{align} \label{symmetry}
\Eehat [f(T,\rho, \xi)g(T,v_0,\xi)\xi(\rho,v_0)] 
= \Eehat [f(T,v_0, \xi)g(T,\rho,\xi)\xi(v_0,\rho)] . 
\end{align}    
The environment seen from the random walk is the process $(T,X_n,\xi)_{n\in \mathbb{Z}}$ with state space $\Omega$ and transition operator
\begin{align}
Gf(T,\rho,\xi) = \frac{1}{\pi(\rho)}\sum_{v\sim \rho} \xi(\rho,v) f(T,v,\xi) . 
\end{align}
Using the symmetry relation \eqref{symmetry}, it is then proven in \cite[Lemma 3.1]{gantert2012random} that $G$ is reversible with respect to $\Pehat$, 
\begin{align} \label{reversible}
\Eehat [f(T,\rho, \xi)Gg(T,v_0,\xi)] 
= \Eehat [Gf(T,\rho, \xi)g(T,v_0,\xi)]  
\end{align}    
for $f,g\in L^2(\Pehat)$, in particular, $\Pehat$ is an invariant measure for the environment seen from the random walk. In the reference, the invariance is proven for uniformly elliptic conductances, but it works for $\varepsilon=0$ as well.   

In order to prove Theorem \ref{thm:main}, we need to consider $\Pehat$ as $\varepsilon\to 0$. Note that all the measures introduced in this section are well defined also if $\varepsilon=0$, i.~e., some edge have conductance 0. We then have the following continuity in $\varepsilon$, the proof is in Section \ref{sec:proofs}.

\begin{lem} \label{lem:convinvariantmeasure}
As $\varepsilon\to 0$, we have $\Pehat \to \hat{\mathrm{P}}_0$ and $\PPehat \to \hat{\mathbb{P}}_0$ weakly.
\end{lem}

\begin{remark} \label{rem:invariant0}
Let us comment on the limit measure $\hat{\mathrm{P}}_0$. If $\varepsilon=0$, the edges with positive conductance form a forest of disconnected trees. Since the random walk can only traverse edges with positive conductance, it moves on a subgraph of the original tree $T$. The measure $\hat{\mathrm{P}}_0$ is still invariant for the environment seen from the random walk. It is however not the unique invariant measure anymore. For example, any measure under which $\{\pi(\rho)=0\}$ has probability 1 is trivially invariant. Since the measure $\hat{\mathrm{P}}_0$ is the weak limit of $\Pehat$, it is the right measure to consider for the limit of the speed. Note that the event $\{\pi(\rho)=0\}$ has probability 0 under $\hat{\mathrm{P}}_0$, so there is always at least one edge with positive conductance adjacent to $\rho$. Indeed, for $\varepsilon$ small the vertices $v$ with $\pi(v)=0$ can only be reached by a $\varepsilon$-edge, which happens rarely, but they don't trap the random walk: the next step from $v$ is chosen uniformly among its neighbors. Therefore, the fraction of time spend in such vertices vanishes as $\varepsilon\to 0$.    
\end{remark}

For vertices $u,v,z\in T$, we define the signed distance of $u$ to $v$ with respect to $z$ as
\begin{align} \label{def:horodistance}
[u-v]_z:= d_T(u,z)-d_T(v,z),
\end{align}
where $d_T(u,z)$ is the graph distance between vertices $u$ and $z$. Note that this notion of distance is additive,
\begin{align}
[u-v]_z + [v-w]_z = [u-w]_z .
\end{align} 
The distance $[u-v]_z$ does not change when the reference vertex $z$ is replaced by a vertex $\tilde z$, provided that the path from $z$ to $\tilde z$ is disjoint from the paths from $z$ to $u$ and from $z$ to $v$. This implies that if $(x_m)_{m\geq 0}$ is a transient path on $T$, $[u-v]_{x_m}$ is constant for $m$ large enough and the following limit is well defined:
\begin{align} \label{def:horodistance2}
[u-v]_{x_{-\infty}} := \lim_{m\to \infty} [u-v]_{x_m}. 
\end{align}
The limit $[u-v]_{x_{-\infty}}$ is called the horodistance of $u$ to $v$, relative to the boundary point $x_{-\infty}$. As observed in \cite{LyoPemPer96a}, and applied to our setting by \cite{gantert2012random}, the ergodicity of the environment seen from the particle under $\PPehat$ implies that the speed is given as
\begin{align} \label{speedviainvmeasure}
\lim_{n\to \infty} \frac{|X_n|}{n} = \EEehat \big[ [X_1-X_0]_{X_{-\infty}}\big] ,
\end{align}
where the limit holds $\mathbb{P}_\varepsilon$-almost surely, or equivalently $\PPehat$-almost surely, for any $\varepsilon>0$. Note that the distance $[X_1-X_0]_{X_{-\infty}}$ is only well-defined for $\varepsilon>0$. If $\varepsilon=0$, the random walk might start on a tree which is finite when edges with conductance zero are removed. In this case, the random walk in negative time cannot be transient and does not define a boundary point of the tree. We show in the following section that the limit of the speed for $\varepsilon\to 0$ is given by the corresponding limit of \eqref{speedviainvmeasure}, on the event that the tree $T_0$ formed by edges with positive conductance is infinite, that is,
\begin{align}\label{speedlimitnew}
\lim_{\varepsilon \searrow 0} v(\nu,\mu_\varepsilon) = \hat{\mathbb{E}}_0 \big[ [X_1-X_0]_{X_{-\infty}}\mathbbm{1}_{\{|T_0|=\infty\}} \big] .
\end{align}   
On the event $\{|T_0|=\infty\}$ the limit $X_{-\infty}$ is well defined. This implies Theorem \ref{thm:main}, since 
\begin{align}
\hat{\mathbb{P}}_0(|T_0|=\infty) = \hat{\mathrm{P}}_0(|T_0|=\infty) 
\end{align}
and
\begin{align}
\hat{\mathbb{E}}_0 \big[ [X_1-X_0]_{X_{-\infty}}\mid |T_0|=\infty \big] = v(\nu,\mu_0) .
\end{align}

\section{Proofs of the main results}
\label{sec:proofs}

\subsection{Proof of Proposition \ref{prop:cont}}

It was proven by \cite{gantert2012random}, that the speed can be expressed as an expectation of ratios of effective conductances. Let $\mathcal C(T)$ be the effective conductance from the root to infinity, then
\begin{align}
v(\nu,\mu) = 1-\frac2\gamma \mathrm{E^{aug}}\left[\xi_0\frac{\mathcal C(T^*)}{\mathcal C(T)}\right],\label{speedgantert}
\end{align} 
where $\xi_0=\xi(\rho,v_0)$ denotes the conductance of the additional edge in the augmented Galton-Watson tree and $T^*$ denotes the subtree composed of the tree rooted at $v_0$ and the additional edge $(\rho,v_0)$.

Now, let $(\mu_n)_{n\in\mathbb N}$ and $\mu$ be uniformly elliptic, i.~e. satisfying \eqref{uniformelliptic} for a common $\delta>0$, such that $\mu_n\to\mu$ weakly. We denote the law of the augmented Galton-Watson tree by $\mathrm P_n^{\mathrm{aug}}$ (respectively $\mathrm{P^{aug}}$), if the marginal law of the conductances is given by $\mu_n$ (respectively $\mu$). Then, \eqref{speedgantert} implies
\begin{align}
\lim_{n\to\infty} v(\nu,\mu_n)
&=1-\frac2\gamma \lim_{n\to \infty}\mathrm{E}_n^{\mathrm{aug}}\left[\xi_0\frac{\mathcal C(T^*)}{\mathcal C(T)}\right] \notag
\\&=1-\frac2\gamma \lim_{n\to \infty}\mathrm{E}_n^{\mathrm{aug}}\left[\mathrm{E}_n^{\mathrm{aug}}\left[\xi_0\frac{\mathcal C(T^*)}{\mathcal C(T)}\middle|T\right]\right] \notag
\\&=1-\frac2\gamma \lim_{n\to \infty}\mathrm{E}^{\mathrm{aug}}\left[\mathrm{E}_n^{\mathrm{aug}}\left[\xi_0\frac{\mathcal C(T^*)}{\mathcal C(T)}\middle|T\right]\right], \label{limitofspeed}
\end{align}
where for the last equality, we used that the distribution of $T$ is independent of $n$. 
We first show that the sequence of the conditional expectations in \eqref{limitofspeed} convergences in distribution, 
\begin{align}
\mathrm{E}_n^{\mathrm{aug}}\left[\xi_0\frac{\mathcal C(T^*)}{\mathcal C(T)}\middle|T\right] 
\xrightarrow[n\to \infty]{d}
\mathrm{E}^{\mathrm{aug}}\left[\xi_0\frac{\mathcal C(T^*)}{\mathcal C(T)}\middle|T\right]. \label{weakconv_conditionalexpectation}
\end{align}
To see this, we change the probability space by defining a tree $T$ as a subset of the Ulam-Harris tree and labeling each edge of the Ulam-Harris tree (not only the edges in a particular tree) with conductances. This entails that the conductances are independent of the tree, but the considered expectations remain the same.
Since $\mu_n\to \mu$ weakly by assumption, the weak convergence $\mathrm{P}_n\to \mathrm{P}$ follows and then also $\mathrm{P}^\mathrm{aug}_n\to \mathrm{P}^\mathrm{aug}$ by dominated convergence. Hence, \eqref{weakconv_conditionalexpectation} follows from \cite[Corollary 4.1]{crimaldi2005convergence}, if the mapping $\xi \mapsto \xi_0\frac{\mathcal C(T^*)}{\mathcal C(T)}$ is bounded and continuous.
Boundedness follows from $\mathcal{C}(T^*)\le\mathcal C(T)$ and the uniform ellipticity of $\mu_n$. Dirichlet's Principle implies that the effective conductance can be expressed as the minimum of a linear function in $\xi$, thus it is concave. Furthermore, for almost all $T,T^*$, the simple random walks on $T,T^*$ are transient, which implies that $\xi \mapsto C(T)$ and $\xi \mapsto C(T^*)$ are locally bounded on the convex set $\{(\xi_e)_e:\xi_e>0\}$. Therefore (see \cite[Section 41]{RobVar73}), both effective conductances are continuous on $\{(\xi_e)_e:\xi_e>0\}$ for almost all trees, which implies \eqref{weakconv_conditionalexpectation}. 

Now, note that the conditional expectation in \eqref{limitofspeed} is bounded by $\delta^{-1}$ for all $n$, thus it is uniformly integrable. So, the sequence of conditional expectations converges in expectation and we obtain from \eqref{limitofspeed}
\begin{align*}
\lim_{n\to\infty} v(\nu,\mu_n)
=1-\frac2\gamma \mathrm{E}^{\mathrm{aug}}\left[\mathrm{E}^{\mathrm{aug}}\left[\xi_0\frac{\mathcal C(T^*)}{\mathcal C(T)}\middle|T\right]\right]
=1-\frac2\gamma \mathrm{E}^{\mathrm{aug}}\left[\xi_0\frac{\mathcal C(T^*)}{\mathcal C(T)}\right]
=v(\nu,\mu),
\end{align*}
which concludes the proof of Proposition~\ref{prop:cont}.

\subsection{Proof of Theorem \ref{thm:main}}

As mentioned in \eqref{speedviainvmeasure}, for uniformly elliptic environment measures, the speed is the expectation under the invariant measure $\PPehat$ of $D_\infty=[X_1-X_0]_{X_{-\infty}}$. 
The weak convergence of $\PPehat$ does not imply the convergence of the speed, because $D_{\infty}$ is not a continuous function of the trajectory of the bi-infinite random walk. Also, note that since $T_0$ might be finite, $D_\infty$ is not well-defined for $\varepsilon=0$. On the event $\{|T_0|=\infty\}$ however, $D_\infty$ is well-defined even if $\varepsilon=0$ and we therefore distinguish whether $T_0$ is a finite tree or not. For $\varepsilon>0$, the tree $T_0$ is defined as the subgraph formed by the edges with conductance larger than $\varepsilon$ containing the root. The proof of Theorem \ref{thm:main} follows from the two following lemmas. 

\begin{lem}\label{lem:speedoninfinitetrees}
\begin{align*}
\lim_{\varepsilon\to 0} \EEehat \left[ D_\infty \mathbbm{1}_{\{|T_0|=\infty\}} \right]
= \hat{\mathbb{E}}_0\left[ D_\infty \mathbbm{1}_{\{|T_0|=\infty\}} \right]
\end{align*}
\end{lem} 

\begin{lem}\label{lem:speedonfinitetrees}
\begin{align*}
\lim_{\varepsilon\to 0} \EEehat \left[ D_\infty \mathbbm{1}_{\{|T_0|<\infty\}} \right] =0
\end{align*}
\end{lem} 

Combining Lemma \ref{lem:speedoninfinitetrees} and Lemma \ref{lem:speedonfinitetrees}, we obtain
\begin{align*}
\lim_{\varepsilon\to 0} v(\nu,\mu_\varepsilon)= \lim_{n\to \infty} \EEehat[D_\infty] = \hat{\mathbb{E}}_0\left[ D_\infty \mathbbm{1}_{\{|T_0|=\infty\}} \right] = \hat{\mathbb{E}}_0\left[ D_\infty | |T_0|=\infty  \right] \hat{\mathrm{P}}_0 (|T_0|=\infty) 
\end{align*}
and as described in the introduction, $\hat{\mathbb{E}}_0\left[ D_\infty | |T_0|=\infty  \right]=v(\nu,\mu_0)$.

\subsubsection{Proof of Lemma \ref{lem:speedoninfinitetrees}}

The main ingredients of the proof are the continuity of $\PPehat$ as $\varepsilon\to 0$ and a transience estimate on infinite subtrees $T_0$, which is uniform in $\varepsilon$. 
 
As an approximation for $D_{\infty}$ which depends only on finitely many coordinates, we introduce
\begin{align*}
D_M:= [X_1-X_0]_{{X}_{-M}} = d_T(X_1,X_{-M})-d_T(X_0,X_{-M}) .
\end{align*}
Since $D_M$ depends only on finitely many coordinates, it is a continuous and uniformly bounded function of the trajectory. Moreover, we approximate $\mathbbm{1}_{\{|T_0|=\infty\}}$ by the indicator functions $\mathbbm{1}_{\{|T_0|>N\}}$, which are continuous functions of the environment. For the limit of the speed we write then
\begin{align} \label{limitDM1}
\lim_{\varepsilon\to 0} \EEehat \left[D_\infty \mathbbm{1}_{\{|T_0|=\infty\}} \right]
= \lim_{\varepsilon\to 0} &\big(\EEehat \left[D_M \mathbbm{1}_{\{|T_0|>N\}} \right]
+ \EEehat \left[D_M \left(\mathbbm{1}_{\{|T_0|=\infty\}}-\mathbbm{1}_{\{|T_0|>N\}} \right)\right] \notag
\\&\quad+\EEehat \left[(D_\infty-D_M)\mathbbm{1}_{\{|T_0|=\infty\}}\right]\big) .
\end{align}
By definition \eqref{def:horodistance2}, we have $D_M\to D_{\infty}$ as $M\to \infty$ almost surely on any infinite tree. The weak convergence in Lemma \ref{lem:convinvariantmeasure} and dominated convergence yield then for the first term on the right hand side of \eqref{limitDM1}
\begin{align} \label{limitDM2}
\lim_{M\to \infty} \lim_{N\to\infty} \lim_{\varepsilon\to 0}  \EEehat \left[ D_M \mathbbm{1}_{\{|T_0|>N\}} \right]
=\lim_{M\to \infty} \lim_{N\to\infty} \hat{\mathbb{E}}_0\left[ D_M \mathbbm{1}_{\{|T_0|>N\}} \right]
= \hat{\mathbb{E}}_0\left[ D_\infty \mathbbm{1}_{\{|T_0|=\infty\}} \right] .
\end{align}
We can bound the absolute value of the second expectation in \eqref{limitDM1} by
\begin{align} \label{limitDM4}
\left| \EEehat \left[D_M \left(\mathbbm{1}_{\{|T_0|=\infty\}}-\mathbbm{1}_{\{|T_0|>N\}} \right)\right] \right|
&\le 2\Eehat \left[ \mathbbm{1}_{\{|T_0|>N\}} - \mathbbm{1}_{\{|T_0|=\infty\}} \right].
\end{align}
Recall the definition of $\Pehat$ in \eqref{defPhat}, we have
\begin{align*}
\Eehat \left[ \mathbbm{1}_{\{|T_0|>N\}} - \mathbbm{1}_{\{|T_0|=\infty\}} \right]
&= \mathrm{E}_\varepsilon^{\mathrm{aug}} \left[ \left(\mathbbm{1}_{\{|T_0|>N\}}-\mathbbm{1}_{\{|T_0|=\infty\}} \right)\frac{\pi(\rho)}{\gamma_\varepsilon \mathrm{deg}(\rho)}\right]
\\&=\mathrm{E}_\varepsilon^{\mathrm{aug}} \left[ \left(\mathbbm{1}_{\{|T_0|>N\}}-\mathbbm{1}_{\{|T_0|=\infty\}} \right)\frac{1}{\gamma_\varepsilon \mathrm{deg}(\rho)} \mathrm{E}_\varepsilon^{\mathrm{aug}} \left[\pi(\rho)\middle|T\right] \right]
\\&=\mathrm{E}_\varepsilon^{\mathrm{aug}} \left[ \mathbbm{1}_{\{|T_0|>N\}}-\mathbbm{1}_{\{|T_0|=\infty\}} \right]
\\&=\mathrm{E}_0^{\mathrm{aug}} \left[ \mathbbm{1}_{\{|T_0|>N\}}-\mathbbm{1}_{\{|T_0|=\infty\}} \right].
\end{align*}
To see the last equality, observe that the distribution of the indicator functions does not depend on $\varepsilon$, since $T_0$ is the subtree formed by the edges with conductance larger than $\varepsilon$. Plugging this in \eqref{limitDM4}, we obtain
\begin{align}
\lim_{N\to\infty} \lim_{\varepsilon\to 0} 
\left| \EEehat \left[D_M \left(\mathbbm{1}_{\{|T_0|=\infty\}}-\mathbbm{1}_{\{|T_0|>N\}} \right)\right] \right| 
\le \lim_{N\to\infty} \mathrm{E}_0^{\mathrm{aug}} \left[ \mathbbm{1}_{\{|T_0|>N\}}-\mathbbm{1}_{\{|T_0|=\infty\}} \right]
=0.\label{limitDM5}
\end{align}
In view of \eqref{limitDM1} and the limits \eqref{limitDM2} and \eqref{limitDM5}, the proof is completed once we show 
\begin{align} \label{limitDM3}
\lim_{M\to \infty } \limsup_{\varepsilon\to 0} \EEehat \left[|D_\infty-D_M|\mathbbm{1}_{\{|T_0|=\infty\}}\right] = 0 .
\end{align}
In fact, we may replace the event $\{|T_0|=\infty\}$ by the event $\{X_{-M}\in T_0, |T_0|=\infty\}$, since $|D_\infty-D_M|\leq 2$ and $\PPehat(X_{-M}\notin T_0)$ vanishes as $\varepsilon\to 0$ for $M$ fixed. 
 Since $D_M=D_\infty$ on the event that $X_{-k}\neq \rho $ for all $k\geq M$, we can bound
\begin{align}
\EEehat\left[|D_\infty-D_M|\mathds 1_{\{X_{-M}\in T_0, |T_0|=\infty\}}\right] & \leq 2 \PPehat(X_{-k}=\rho \text{ for some } k\geq M, X_{-M}\in T_0,|T_0|=\infty) .
\end{align}
We need to show that the right hand side vanishes as $M\to\infty$, uniformly in $\varepsilon$. That is, we need to bound the return time to the root after a large time $M$, uniformly in $\varepsilon$. 

We denote by $\mathcal{C}_\omega(v,A)$ the effective conductance between  vertex $v$ and a set $A$ of vertices in the environment $\omega$, see \cite{doyle1984random} or \cite{LyoPer16}. Let $\eta(A)$ denote the hitting time of the set $A$, then (see \cite[Exercise 2.34]{LyoPer16}) we have for a random walk starting at $v$
\begin{align*}
P_\omega ( \eta(A)<\eta(B) ) \leq \frac{\mathcal{C}_\omega(v,A)}{\mathcal{C}_\omega(v,B)} . 
\end{align*}
This yields in our situation
\begin{align}\label{electricbound}
\hat P_\omega (X_{-k}=\rho \text{ for some } k\geq M |X_{-M}=v) \leq \frac{\mathcal{C}_\omega(v,\rho)}{\mathcal{C}_\omega(v,\infty)}  , 
\end{align} 
which implies 
\begin{align}
\PPehat(X_{-k}=\rho \text{ for some } k\geq M,X_{-M}\in T_0, |T_0|=\infty)
\leq \EEehat \left[ \frac{\mathcal{C}_\omega(X_{-M},\rho)}{\mathcal{C}_\omega(X_{-M},\infty)} \mathbbm{1}_{\{X_{-M}\in T_0,|T_0|=\infty\}} \right] . 
\end{align}
The Cauchy-Schwarz inequality implies the bound
\begin{align} \label{electric+CS}
& \PPehat(X_{-k}=\rho \text{ for some } k\geq M, X_{-M}\in T_0,|T_0|=\infty)^2 \notag \\
& \qquad \leq \EEehat \left[ \mathcal{C}_\omega(X_{-M},\rho)^2\mathbbm{1}_{\{|T_0|=\infty\}} \right] \EEehat \left[\mathcal{C}_\omega(X_{-M},\infty)^{-2} \mathbbm{1}_{\{X_{-M}\in T_0,|T_0|=\infty\}} \right] . 
\end{align}
For the first expectation, we note that the effective conductance between the root and $X_{-M}$ depends only on the path between these two vertices, so that for any $N\geq 1$, $\hat E_\omega[\mathcal{C}_\omega(X_{-M},\rho)]\mathbbm{1}_{\{|T_0|>N\}}$ is a continuous function of $\omega$ and bounded by $(M\delta)^{-1}$. Lemma \ref{lem:convinvariantmeasure} implies then
\begin{align}
\limsup_{\varepsilon\to 0} \EEehat \left[ \mathcal{C}_\omega(X_{-M},\rho)^2\mathbbm{1}_{\{|T_0|=\infty\}} \right] & \leq 
\limsup_{\varepsilon\to 0} \EEehat \left[ \mathcal{C}_\omega(X_{-M},\rho)^2\mathbbm{1}_{\{|T_0|>N\}} \right] \notag \\
& = \hat{\mathbb{E}}_0\left[ \mathcal{C}_\omega(X_{-M},\rho)^2\mathbbm{1}_{\{|T_0|>N\}} \right] .
\end{align}  
Letting $N\to\infty$, we have by monotone convergence
\begin{align}
\limsup_{\varepsilon\to 0} \EEehat \left[ \mathcal{C}_\omega(X_{-M},\rho)^2\mathbbm{1}_{\{|T_0|=\infty\}} \right] \leq \hat{\mathbb{E}}_0\left[ \mathcal{C}_\omega(X_{-M},\rho)^2\mathbbm{1}_{\{|T_0|=\infty\}} \right] .
\end{align}
Since under $\hat{\mathbb{P}}_0$, conditioned on $|T_0|=\infty$, the random walk is transient and the conductance between $X_{-M}$ and $\rho$ is bounded by $(|X_{-M}|\delta)^{-1}$, the conductance $\mathcal{C}_\omega(X_{-M},\rho)$ converges almost surely to zero. Dominated convergence implies then
\begin{align} \label{ersterEwert}
\lim_{M\to \infty} \limsup_{\varepsilon\to 0} \EEehat \left[ \mathcal{C}_\omega(X_{-M},\rho)^2\mathbbm{1}_{\{|T_0|=\infty\}} \right] = 0 .
\end{align}  
We now show that the second expectation in \eqref{electric+CS} remains bounded. Let us denote by $\mathcal{R}(\rho,\infty)=\mathcal{C}_\omega(\rho,\infty)^{-1}$ the effective resistance. Since the process $(T,X_n,\xi)_{n\in \mathbb{Z}}$ is stationary under $\PPehat$, 
\begin{align}
\EEehat \left[\mathcal{R}_\omega(X_{-M},\infty)^{2} \mathbbm{1}_{\{X_{-M}\in T_0,|T_0|=\infty\}} \right] 
& = \EEehat \left[\mathcal{R}_\omega(\rho,\infty)^{2} \mathbbm{1}_{\{X_{M}\in T_0,|T_0|=\infty\}} \right] \notag \\
& \leq \Eehat \left[\mathcal{R}_\omega(\rho,\infty)^{2} \mathbbm{1}_{\{|T_0|=\infty\}} \right] .
\end{align}
By Rayleigh's monotonicity principle, the effective resistance $\mathcal{R}_\omega(\rho,\infty)$ increases when all edges with conductance $\varepsilon$ are removed, which implies
\begin{align}\label{resistancemoment}
\Eehat \left[\mathcal{R}_\omega(\rho,\infty)^{2} \mathbbm{1}_{\{|T_0|=\infty\}} \right] \leq \hat{\mathrm{E}}_0 \left[\mathcal{R}_\omega(\rho,\infty)^{2} \mathbbm{1}_{\{|T_0|=\infty\}} \right] \leq \hat{\mathrm{E}}_0 \left[\mathcal{R}_\omega(\rho,\infty)^{2} \mid |T_0|=\infty  \right] .
\end{align}
By the Harris decomposition (see \cite[Proposition 4.10]{Lyon92}), the tree $T_0$,  conditioned on $|T_0|=\infty$,  may be decomposed into a backbone tree without leaves and with finite trees attached to it. The effective resistance $\mathcal{R}_\omega(\rho,\infty)$ only depends on the backbone tree, which is again a supercritical Galton-Watson tree. We may then apply Lemma \ref{lem:resistancemoments} below, which shows that the second moment in \eqref{resistancemoment} is finite. 

Together with \eqref{ersterEwert}, this implies that the upper bound in \eqref{electric+CS} vanishes when first $\varepsilon$ tends to zero and then $M$ to infinity. This yields \eqref{limitDM3} and concludes the proof of the lemma. 
\hfill $\Box$

\medskip

The following lemma shows the moment bound needed for the proof of Lemma \ref{lem:speedonfinitetrees}. It shows that the effective resistance of Galton-Watson trees has finite moments of any order, which might be of independent interest. 

\begin{lem}\label{lem:resistancemoments}
	Let $T$ be a Galton-Watson tree without leaves, offspring mean $m>1$ and uniformly elliptic conductances. Then for all $p>0$
	\begin{align*}
	\mathrm{E}\left[\mathcal{R}_\omega(\rho, \infty)^p\right]<\infty .
	\end{align*} 
\end{lem}

\begin{proof}
	Let $G_n(T)=\{v\in T:|v|=n\}$ be the $n$-th generation of the tree $T$. We have
	\begin{align*}
	\mathcal R_\omega(\rho,\infty)=\lim_{n\to\infty} \mathcal R_\omega(\rho,G_n(T)).
	\end{align*}
	According to the assumption of the uniformly ellipticity of the conductances there exists some $\delta>0$ such that $\xi(e)\ge \delta$ almost surely. By Rayleigh's monotonicity principle, the effective resistance $\mathcal{R}_\omega(\rho,G_n(T))$ increases when the conductances of all edges are reduced to $\delta$. 
    Let $N$ denote the number of offspring of the root and $v_1,\dots,v_N$ be the vertices of the first generation.
	Moreover, call $T_{v}$ the subtree of $T$ rooted at $v$ composed of $v$ and all its descendants.
	Using the Parallel Law and the Series Law, we obtain the following recursion
	\begin{align*}
	&\mathcal R_\omega(\rho,G_n(T))	
	\leq \mathcal R_{(T,\rho,(\delta)_e)}(\rho,G_n(T)) \\&\quad=\left(\frac{1}{\delta^{-1} + \mathcal R_{(T_{v_1},\rho,(\delta)_e)}(v_1,G_{n-1}(T_{v_1}))}+\dots + \frac{1}{\delta^{-1} + \mathcal R_{(T_{v_N},\rho,(\delta)_e)}(v_N,G_{n-1}(T_{v_N}))}\right)^{-1}.
	\end{align*}
	For the effective resistance between $v$ and $G_n(T_{v})$ we write $\mathcal R_n=\mathcal R_{(T_v,v,(\delta)_e)}(v,G_n(T_v))$.
	Since the subtrees $T_{v_1},\dots T_{v_N}$ are independent Galton-Watson trees, we obtain the following stochastic domination
	\begin{align*}
	\mathcal R_n \preccurlyeq 
	\left(\frac{1}{\delta^{-1} + \mathcal R_{n-1}^{(1)}}+\dots + \frac{1}{\delta^{-1} + \mathcal R_{n-1}^{(N)}}\right)^{-1},
	\end{align*} 
	where $\mathcal R_{n-1}^{(1)}, \dots,\mathcal R_{n-1}^{(N)}$ denote independent copies of $\mathcal R_{n-1}$.
	Bounding the harmonic mean by the arithmetic mean gives rise to 
	\begin{align}
	\mathcal R_n \preccurlyeq 
	\frac1{N^2} \sum_{i=1}^N \left(\delta^{-1}+\mathcal R_{n-1}^{(i)}\right).\label{effresistance_stochasticdomination}
	\end{align}
	
	Now, write $\mathcal R_\infty= \mathcal R_{(T_v,v,(\delta)_e)}(v,\infty) =\lim_{n\to\infty} \mathcal R_n$ for the effective resistance between $v$ and infinity.
    From \cite[Lemma 9.1]{LyoPemPer96a} follows, that the mean resistance $\mathrm E[\mathcal R_\infty]$ is bounded.
	In the following, we show by induction that all the higher moments exist as well. 
	Therefore, suppose that for a given $m\in\mathbb N$ the moment $\mathrm E[\mathcal R_\infty^{m-1}]$ is finite. Using the stochastic domination in \eqref{effresistance_stochasticdomination}, we have
	\begin{align*}
	\mathrm E[\mathcal R_n^m]
	&\leq
	\mathrm E\left[\left((\delta N)^{-1} + N^{-2} \sum_{i=1}^N \mathcal R_{n-1}^{(i)}\right)^m\right]
	\\&= \mathrm E\left[\sum_{k=0}^m \binom{m}{k}(\delta N)^{k-m}N^{-k} \left(N^{-1} \sum_{i=1}^N \mathcal R_{n-1}^{(i)}\right)^k\right].
	\end{align*}
	Applying Jensen's inequality gives the following upper bound,
	\begin{align}
	\mathrm E[\mathcal R_n^m]
	&\leq \sum_{k=0}^m \binom{m}{k} \delta^{k-m} \mathrm E \left[ N^{-m-1}\sum_{i=1}^N \left(\mathcal R_{n-1}^{(i)}\right)^k\right] \notag
	\\&= \sum_{k=0}^m \binom{m}{k} \delta^{k-m} \mathrm E\left[N^{-m-1} \sum_{i=1}^N \mathrm E\left[\left(\mathcal R_{n-1}^{(i)}\right)^k \middle| N\right]\right] \notag
	\\& = \sum_{k=0}^{m-1} \binom{m}{k} \delta^{k-m}\mathrm E\left[N^{-m}\right]\mathrm E\left[\mathcal R_{n-1}^k\right]
	+ \mathrm E[N^{-m}]\mathrm E[\mathcal R_{n-1}^m] \notag
	\\&\leq \sum_{k=0}^{m-1} \binom{m}{k} \delta^{k-m}\mathrm E\left[N^{-m}\right]\mathrm E\left[\mathcal R_{\infty}^k\right]
	+ \mathrm E[N^{-m}]\mathrm E[\mathcal R_{n-1}^m], \label{effresistance_mthmoment_bound}
	\end{align}
	where we used the monotonicity of the effective resistance in the last inequality. More precisely, we have $\mathcal R_n\le\mathcal R_\infty$ for all $n$. The finiteness of $\mathrm E\left[\mathcal R_\infty^k\right]$ for $k=0,\dots, m-1$ follows from the induction hypothesis.\\
	A vanishing resistance between two vertices involves shorting them together, so that they behave as if they were a single vertex. Consistently to this interpretation, we define $\mathcal R_0=0$.
	Hence, we have a recursion of the from $x_n\leq a+bx_{n-1}$, $x_0=0$, for some $a, b>0$. Iterating this gives rise to $x_n\leq a\sum_{k=0}^{n-1} b^k$.
	In our setting, iterating \eqref{effresistance_mthmoment_bound} leads to
	\begin{align}
	\mathrm E[\mathcal R_n^m]&\leq
	\left( \sum_{k=0}^{m-1} \binom{m}{k} \delta^{k-m}\mathrm E\left[N^{-m}\right]\mathrm E\left[\mathcal R_{\infty}^k\right] \right)
	\sum_{k=0}^{n-1} \mathrm E\left[N^{-m}\right]^k. \label{effresistance_iterating}
	\end{align}
    Since the absolut value of $\mathrm E\left[N^{-m}\right]$ is less than one, \eqref{effresistance_iterating} converges and
	\begin{align*}
	\mathrm E[\mathcal R_\omega(\rho,\infty)^m]
	&\leq \mathrm E[\mathcal R_\infty^m] = \lim_{n\to\infty}\mathrm E[\mathcal R_n^m]
	\\&\leq	\left( \sum_{k=0}^{m-1} \binom{m}{k} \delta^{k-m}\mathrm E\left[N^{-m}\right]\mathrm E\left[\mathcal R_\infty^k\right] \right)
	\cdot \frac{1}{1-\mathrm E[N^{-m}]}<\infty,
	\end{align*}
	which completes the proof.
\end{proof}

\subsubsection{Proof of Lemma \ref{lem:speedonfinitetrees}}

The main ingredients of this proof are the stationarity and the ergodicity of the process $(T, X_n,\xi)_{n\in\mathbb Z}$ under $\mathbb{\hat P}_\varepsilon$. By the ergodic theorem we have
\begin{align}
\hat{\mathbb{E}}_\varepsilon\left[[X_1-X_0]_{X_{-\infty}} \mathbbm{1}_{\{|T_0|<\infty\}}\right]
=\lim_{n\to\infty} \frac1n \sum_{k=0}^{n-1} [X_{k+1}-X_k]_{X_{-\infty}}  \mathbbm{1}_{\{|T_0(X_k)|<\infty\}}\label{ergodictheorem}
\end{align} 
almost surely, where $T_0(v)$ is the subtree formed by edges with conductance larger than $\varepsilon$ and containing $v$. As the sum can only increase when the random walk moves on finite trees, let us define the points in time, when the random walk enters and leaves finite trees,
\begin{align*}
b_1 &=\inf\{n\ge 1:X_n\notin|T_0(X_0)|\},\\
a_{k+1} &=\inf\{n\ge b_k:|T_0(X_n)|<\infty\},\quad k\in\mathbb N,\\
b_{k+1} &=\inf\{n\ge a_k:X_n\notin|T_0(X_{a_k})|\}, \quad k\in\mathbb N.
\end{align*}   
Since the limit in \eqref{ergodictheorem} exists, it is equal to the limit along the subsequence, which consists of the times when the random walk is located in a finite tree,
\begin{align}
\hat{\mathbb{E}}_\varepsilon\left[[X_1-X_0]_{X_{-\infty}} \mathbbm{1}_{\{|T_0|<\infty\}}\right]
&=\lim_{n\to\infty} \frac1{b_n}\sum_{k=0}^{b_1-1} [X_{k+1}-X_k]_{X_{-\infty}}\notag
\\&\quad+\lim_{n\to\infty} \frac1{b_n}\sum_{k=2}^{n}\sum_{i=a_k}^{b_k-1} [X_{i+1}-X_i]_{X_{-\infty}}.\label{alongsubsequence}
\end{align}
We note that the first limit in \eqref{alongsubsequence} vanishes almost surely. Now, let $D_k$ denote the time that the random walk spends in the tree $T_0(X_k)$, that is
\begin{align*}
D_k &= \sup\{n\in\mathbb N:X_{k+1},\dots,X_{k+n}\in T_0(X_k)\}
\\&\quad+ \sup\{n\in\mathbb N:X_{k-1},\dots,X_{k-n}\in T_0(X_k)\}
+1.
\end{align*}
Bounding the distance from the root that the random walks can gain on the $k$-th finite tree by its number of vertices, we obtain the following upper bound for the second limit in \eqref{alongsubsequence}
\begin{align}
\lim_{n\to\infty} \frac1{b_n}\sum_{k=2}^{n}\sum_{i=a_k}^{b_k-1} [X_{i+1}-X_i]_{X_{-\infty}} 
&\le \lim_{n\to\infty} \frac1{b_n} \sum_{k=2}^n |T_0(X_{a_k})|\notag
\\&= \lim_{n\to\infty} \frac1{b_n} \sum_{k=2}^n \sum_{i=a_k}^{b_k-1} \frac1{D_i}|T_0(X_i)|\notag
\\&= \lim_{n\to\infty} \frac1{b_n} \sum_{k=a_2}^{b_n-1}  \frac1{D_k}|T_0(X_k)| \mathbbm{1}_{\{|T_0(X_k)|<\infty\}}\notag
\\&\le \lim_{n\to\infty} \frac1{b_n} \sum_{k=0}^{b_n-1}  \frac1{D_k}|T_0(X_k)| \mathbbm{1}_{\{|T_0(X_k)|<\infty\}}. \label{ergodictheorem2time}
\end{align}
By the ergodic theorem the averages in \eqref{ergodictheorem2time} converge almost surely to their mean
\begin{align*}
\lim_{n\to\infty} \frac1{b_n} \sum_{k=0}^{b_n-1}  \frac1{D_i}|T_0(X_i)| \mathbbm{1}_{\{|T_0(X_k)|<\infty\}}
= \hat{\mathbb E}_\varepsilon \left[\frac1{D_0} |T_0(X_0)| \mathbbm{1}_{\{|T_0(X_0)|<\infty\}}\right].
\end{align*}
Applying the Cauchy-Schwarz inequality gives rise to the following bound
\begin{align}
\hat{\mathbb E}_\varepsilon \left[\frac1{D_0} |T_0(X_0)| \mathbbm{1}_{\{|T_0(X_0)|<\infty\}}\right]^2
\le \hat{\mathbb E}_\varepsilon \left[\frac1{D_0^2}\right] \hat{\mathbb E}_\varepsilon \left[|T_0(X_0)|^2 \mathbbm{1}_{\{|T_0(X_0)|<\infty\}}\right]. \label{csubound}
\end{align} 
Since $D_0\ge b_1$, we can bound the first expectation in \eqref{csubound} by
\begin{align*}
\hat{\mathbb E}_\varepsilon \left[\frac1{D_0^2}\right]
\le \hat{\mathbb E}_\varepsilon \left[\frac1{b_1^2}\right]
= \sum_{k=0}^\infty \frac1{k^2}\hat{\mathbb P}_\varepsilon(b_1=k).
\end{align*}
Dominated convergence implies
\begin{align}
\lim_{\varepsilon\to0} \hat{\mathbb E}_\varepsilon \left[\frac1{D_0^2}\right] 
\le \sum_{k=0}^\infty \frac1{k^2} \lim_{\varepsilon\to 0}\hat{\mathbb P}_\varepsilon(b_1=k). \label{secondinversmomentD_0}
\end{align}
By definition of $b_1$, we have
\begin{align}
\hat{\mathbb P}_\varepsilon(b_1=k)
&= \hat{\mathbb P}_\varepsilon(X_0,\dots,X_{k-1}\in T_0,X_k\notin T_0)\notag
\\&\le \hat{\mathbb P}_\varepsilon(X_0,\dots,X_{k-1}\in T_0,X_k\notin T_0, |T_0|>0)
+ \hat{\mathbb P}_\varepsilon(|T_0|=0).\label{probabilityrwleavesT0}
\end{align}
We distinguish, if the tree $T_0$ consists only of the root, because in that case the random walk can only leave $T_0$ in the first step. But this case occurs with vanishing probability as $\varepsilon\to0$, due to
\begin{align}
\hat{\mathbb P}_\varepsilon(|T_0|=0)
= \frac2\gamma \int \frac{\pi(\rho)}{\operatorname{deg}(\rho)} \mathbbm{1}_{\{|T_0|=0\}}\,\mathrm d\mathrm {P_\varepsilon^{aug}}
=\frac2\gamma \int \frac{\varepsilon \operatorname{deg}(\rho)}{\operatorname{deg}(\rho)} \,\mathrm d\mathrm {P_\varepsilon^{aug}}
=\frac{2\varepsilon}{\gamma}. \label{boundT0=0}
\end{align}
Now, we bound the first probability in \eqref{probabilityrwleavesT0} as follows
\begin{align}
&\hat{\mathbb P}_\varepsilon(X_0,\dots,X_{k-1}\in T_0,X_k\notin T_0, |T_0|>0)\notag
\\&\quad= \sum_{i=1}^\infty \hat{\mathbb P}_\varepsilon(X_0,\dots,X_{k-1}\in T_0,X_k\notin T_0, |T_0|>0|\operatorname{deg}(X_{k-1})=i) \hat{\mathbb P}_\varepsilon(\operatorname{deg}(X_{k-1})=i) \notag
\\&\quad\le \sum_{i=1}^\infty p_{i-1} \hat{\mathbb P}_\varepsilon(X_k\notin T_0|X_0,\dots,X_{k-1}\in T_0, |T_0|>0) \notag
\\&\quad\le \sum_{i=1}^\infty p_{i-1}  \frac{(i-1)\varepsilon}{(i-1)\varepsilon+\delta}.\label{boundrwleavesT0}
\end{align}
To see the last inequality, observe that there is at least one edge with conductance greater than $\delta$ adjacent to the vertex, where the random walk is located after $k-1$ steps. Decreasing the conductance of the other adjoining edges to $\varepsilon$ implies the bound we used above.
Plugging \eqref{boundT0=0} and \eqref{boundrwleavesT0} in \eqref{probabilityrwleavesT0}, dominated convergence implies
\begin{align*}
\lim_{\varepsilon\to0}\hat{\mathbb P}_\varepsilon(b_1=k)
\le \sum_{i=1}^\infty p_{i-1}  \lim_{\varepsilon\to0} \frac{(i-1)\varepsilon}{(i-1)\varepsilon+\delta}
+ \lim_{\varepsilon\to0}\frac{2\varepsilon}{\gamma} =0.
\end{align*}
Hence, due to \eqref{secondinversmomentD_0}, the second inverse moment of $D_0$ vanishes as $\varepsilon\to0$.

To complete the proof, we show that the second expectation in \eqref{csubound} remains bounded.
By the duality principle for branching processes (see \cite[Section 12]{AthNey72}), the tree $T_0$, conditioned on extinction, is a Galton-Watson tree. Since the second moment of the offspring law exists, Lemma~\ref{lem:secondmomenttotalprogeny} below shows that the second moment in \eqref{csubound} is finite, which concludes the proof of Lemma~\ref{lem:speedonfinitetrees}.
\hfill $\Box$

\medskip

The following lemma shows the finiteness of the second moment needed for the proof of Lemma~\ref{lem:speedonfinitetrees}. It shows a formula to compute the second moment of the total progeny of a branching process.

\begin{lem}\label{lem:secondmomenttotalprogeny}
	For a branching process with independent and identically distributed offspring $X$ having offspring mean $E[X]=\mu <1$ and finite second moment $E[X^2]=m_2<\infty$, the second moment of the total progeny is finite, precisely
	\begin{align*}
	\mathrm E\left[T^2\right] = \frac{m_2-\mu^2-\mu+1}{(1-\mu)^3}.
	\end{align*}
\end{lem}

\begin{proof}
	Let $Z_n$ denote the size of the $n$-th generation, that is
	\begin{align*}
	Z_0=1,\quad Z_n =\sum_{i=1}^{Z_{n-1}} X_{n,i},
	\end{align*}
	where $(X_{n,i})_{i,n\in\mathbb N}$ is a doubly infinite array of independent and identically distributed random variables with $X_{n,i}\sim X$ for all $n,i\in\mathbb N$. Then, the total progeny of the branching process is given by $T=Z_0+Z_1+\dots$.
	Monotone convergence implies the following respresentation of second moment of the total progeny
	\begin{align}\label{2momenttotalprogeny}
	\mathrm E[T^2] = \sum_{n=0}^\infty \mathrm E\left[ Z_n^2\right] +2\sum_{m=0}^\infty \sum_{n=m+1}^\infty \mathrm E\left[Z_m Z_n\right].
	\end{align}
	So, we need to compute the second moment of the generation sizes and the mixed moments. A straightforward calculation shows that
\begin{align}\label{recursion2moment}
\mathrm E[Z_n^2]  = m_2\mu^{n-1}-\mu^{n+1}+\mu^2\mathrm E[Z_{n-1}^2] .
\end{align}	
	Iterating \eqref{recursion2moment} leads to (see also \cite[Section 2]{AthNey72})
	\begin{align}
	\mathrm E[Z_n^2] =\frac{m_2\mu^{n-1}(1-\mu^n)-\mu^{n+1}+\mu^{2n}}{1-\mu}.\label{2momentgenerationsize}
	\end{align}
	Next, we derive a recurrence equation for the mixed moments. Using the independence of $(X_{n,i})_{i\in\mathbb N}$ and $Z_{n-1}$, we obtain for $n> m$
	\begin{align*}
	\mathrm E[Z_m Z_n]
	&=\mathrm E\left[Z_m \sum_{i=1}^{Z_{n-1}} X_{n,i}\right]
	=\mathrm E\left[Z_m \sum_{i=1}^{Z_{n-1}} \mathrm E[  X_{n,i}|Z_{1},\dots,Z_{n-1}]\right]
	\\&=\mathrm E\left[Z_m \sum_{i=1}^{Z_{n-1}}\mathrm E[X_{n,i}]\right]
	=\mathrm E\left[Z_{m}Z_{n-1}\right]\mathrm E[X]
	=\mu \mathrm E[Z_{m}Z_{n-1}].
	\end{align*}
	Iterating this gives rise to 
	\begin{align*}
	\mathrm E[Z_mZ_n]=\mu^{n-m} \mathrm E[Z_m^2] 
	\end{align*}
	for $n>m$.
	Plugging this in \eqref{2momenttotalprogeny}, we obtain
	\begin{align*}
	E[T^2]
	&=\sum_{n=0}^\infty \mathrm E[Z_n^2] 
	+2\sum_{m=0}^\infty \mu^{-m}\mathrm E[Z_m^2] \sum_{n=m+1}^\infty \mu^{n}
	\\&=\sum_{n=0}^\infty \mathrm E[Z_n^2]
	+\frac{2\mu}{1-\mu}\sum_{n=0}^\infty \mathrm E[Z_n^2]
	\\&=\frac{1+\mu}{1-\mu} \sum_{n=0}^\infty \mathrm E[Z_n^2].
	\end{align*}
	Together with \eqref{2momentgenerationsize}, we have
	\begin{align*}
	E[T^2]
	=\frac{1+\mu}{(1-\mu)^2} \sum_{n=0}^\infty (m_2\mu^{n-1}(1-\mu^n)-\mu^{n+1}+\mu^{2n}),
	\end{align*}
	which is finite, since $\mu<1$ holds by assumption. More precisely, we have
	\begin{align*}
	E[T^2]
	=\frac{1+\mu}{(1-\mu)^2}\left( \frac{m_2}{\mu(1-\mu)} - \frac{m_2}{\mu(1-\mu^2)} - \frac{\mu}{1-\mu} + \frac{1}{1-\mu^2} \right) = \frac{m_2-\mu^2-\mu+1}{(1-\mu)^3}.
	\end{align*}
\end{proof}

\subsection{Proof of Lemma \ref{lem:convinvariantmeasure}}

We first show the weak convergence of $\PPehat$ on $\F_m\otimes\G$ with $\F_m=\sigma(X_{-m},\dots,X_m)$. Note that since $\Pehat$ is the marginal law of $\omega$ under $\PPehat$, the weak convergence of $\Pehat$ is implied by the one of $\PPehat$. Let $f$ be a bounded and continuous function, measurable with respect to $\F_m\otimes\G$. 
Let us denote by $T_{| B_m(\rho)}$ the subgraph of $T$ consisting of the vertices of $T$ with graph distance to $\rho$ at most $m$, and the edges connecting those vertices. We have
\begin{align}
\left|\int f \,\D \PPehat - \int f\,\D\hat{\mathbb{P}}_0 \right| 
&\le  \left|\int f \mathds 1_{\{|T_{| B_m(\rho)}|\le M\}} \,\D \PPehat - \int f\mathds 1_{\{|T_{| B_m(\rho)}|\le M\}}\,\D\hat{\mathbb{P}}_0 \right| \label{1}\\
&\quad +  
\left|\int f \mathds 1_{\{|T_{| B_m(\rho)}|> M\}} \,\D \PPehat\right| 
+ \left|\int f \mathds 1_{\{|T_{| B_m(\rho)}|> M\}} \,\D \hat{\mathbb{P}}_0\right|.\notag
\end{align}
Using the Doob-Dynkin lemma we have $f((x_n)_{n\in \mathbb{Z}},\omega)=g((x_{-m},\dots,x_m),\omega)$ for some measurable function $g$. If $|T_{| B_m(\rho)}|\le M$ holds, the number of possible trajectories on $T_{| B_m(\rho)}$ of length $m$ is finite, and we can write the first integral of \eqref{1} as the finite sum
\begin{align} \label{finitetreefinitesum}
& \int f((x_n)_{n\in\mathbb Z},\omega) \mathds 1_{\{|T_{| B_m(\rho)}|\le M\}} \,\D \PPehat \notag \\
&\quad= 
\sum_{x_{-m},\dots,x_m} \int g((x_{-m},\dots,x_m),\omega) \mathds 1_{\{|T_{| B_m(\rho)}|\le M\}} \notag \\
& \quad \quad \quad \quad \times \frac{\pi(\rho)}{\gamma_\varepsilon\mathrm{deg}(\rho)} \hat P_\omega(X_{-m}=x_{-m},\dots, X_m=x_m) \,\D \mathrm{P}^\mathrm{aug}_\varepsilon(\omega) ,
\end{align} 
recall the definition of $\Pehat$ in \eqref{defPhat}. 

Under $\mathrm{P}_\varepsilon$, and also under $\mathrm{P}^\mathrm{aug}_\varepsilon$, the conductances are i.i.d. conditioned on the tree $T$, with marginal law $\mu_\varepsilon$. Since $\mu_\varepsilon\to \mu_0$ weakly, the weak convergence $\mathrm{P}_\varepsilon\to \mathrm{P}_0$ follows and then also $\mathrm{P}^\mathrm{aug}_\varepsilon\to \mathrm{P}^\mathrm{aug}_0$ by dominated convergence. This implies the convergence of the integrals in \eqref{finitetreefinitesum}, provided that the integrands are continuous and bounded functions. 
To verify this, we first look at the mapping $\omega \mapsto \frac{\pi(\rho)}{\mathrm{deg}(\rho)} \hat P_\omega(X_{-m}=x_{-m},\dots, X_m=x_m)$. 

Let $\omega^{(n)} = \left(T^{(n)},\rho^{(n)},\xi^{(n)}\right)$ be a sequence of environments that converges to $\omega=(T,\rho,\xi)$ for $n\to\infty$, that is, $\left(T^{(n)}_{| B_m(\rho^{(n)})},\rho^{(n)}\right)=(T_{| B_m(\rho)},\rho)$ for $n$ big enough and
\begin{align*}
\left(\xi^{(n)}(e)\right)_{e\in\calE(T_{| B_m(\rho)})}\to(\xi(e))_{e\in\calE(T_{| B_m(\rho)})}.
\end{align*} 
For a valid trajectory $x_{-m},\dots,x_m$ (in particular, $x_0=\rho$), we have
\begin{align*}
&\hat P_{\omega^{(n)}}(X_{-m}=x_{-m},\dots, X_m=x_m) 
\\&\quad=  P_{\omega^{(n)}}(X_0=x_0,\dots, X_m=x_m)
 P_{\omega^{(n)}}(X_0=x_0,\dots, X_{-m}=x_{-m})  
\\&\quad= \prod_{k=1}^m \frac{\xi^{(n)}(x_{k-1},x_k)}{\sum_{v\sim x_{k-1}} \xi^{(n)}(x_{k-1},v)} \prod_{k=1}^m \frac{\xi^{(n)}(x_{-(k-1)},x_{-k})}{\sum_{v\sim x_{-(k-1)}} \xi^{(n)}(x_{-(k-1)},v)} .
\end{align*}
Suppose that the sums $\sum_{v\sim x_{k-1}} \xi(x_{k-1},v)$ and $\sum_{v\sim x_{-(k-1)}} \xi^{(n)}(x_{-(k-1)},v)$ are strictly positive for all $k$ appearing in the products. In this case, due to the convergence of $\xi^{(n)}$ we easily obtain 
\begin{align*}
\frac{\pi(\rho^{(n)})}{\mathrm{deg}\left(\rho^{(n)}\right)} \hat P_{\omega^{(n)}}(X_{-m}=x_{-m},\dots, X_m=x_m)
\xrightarrow[n\to \infty]{}
\frac{\pi(\rho)}{\mathrm{deg}(\rho)} \hat P_\omega(X_{-m}=x_{-m},\dots, X_m=x_m) .
\end{align*} 
Now, suppose that there exists a $x_{k-1}$ along the path in positive time with \linebreak $\sum_{v\sim x_{k-1}} \xi(x_{k-1},v)=0$. We define 
\begin{align}
k_0=\min\bigg\{k:\sum_{v\sim x_{k-1}} \xi(x_{k-1},v)=0\bigg\}
\end{align}
as the first time that the sum is zero.  
If $k_0>1$, we notice that the conductance $\xi(x_{{k_0} -2}, x_{{k_0}-1})$ is equal to zero, which implies that $\xi^{(n)}(x_{k_0-2},x_{k_0-1})$ converges to zero. Using that by definition of $k_0$ the sum $\sum_{v\sim x_{k_0-2}} \xi(x_{k_0-2},v)$ is strictly positive, we obtain
\begin{align*}
0\le \prod_{k=1}^m \frac{\xi^{(n)}(x_{k-1},x_k)}{\sum_{v\sim x_{k-1}} \xi^{(n)}(x_{k-1},v)}
\le \frac{\xi^{(n)}(x_{k_0-2},x_{k_0-1})}{\sum_{v\sim x_{k_0-2}} \xi^{(n)}(x_{k_0-2},v)}\xrightarrow[n\to \infty]{} 0.
\end{align*}
If $k_0=0$, we have $\sum_{v\sim \rho} \xi(\rho,v)=\pi(\rho)=0$. In particular, we have $\xi^{(n)}(\rho,x_1)\to \xi(\rho,x_1)=0$ and therefore
\begin{align*}
0\le \frac{\pi(\rho^{(n)})}{\mathrm{deg}\left(\rho^{(n)}\right)} \prod_{k=1}^m \frac{\xi^{(n)}(x_{k-1},x_k)}{\sum_{v\sim x_{k-1}} \xi^{(n)}(x_{k-1},v)}
\le \frac{\pi(\rho^{(n)})}{\mathrm{deg}\left(\rho^{(n)}\right)} \frac{\xi^{(n)}(\rho,x_1)}{\sum_{v\sim \rho} \xi^{(n)}(\rho,v)}
= \frac{\xi^{(n)}(\rho,x_1)}{\mathrm{deg}\left(\rho^{(n)}\right)}
\xrightarrow[n\to \infty]{} 0 .
\end{align*}
In both cases
\small
\begin{align*}
\frac{\pi(\rho^{(n)})}{\mathrm{deg}\left(\rho^{(n)}\right)} \hat P_{\omega^{(n)}}(X_{-m}=x_{-m},\dots, X_m=x_m)
\xrightarrow[n\to \infty]{} 0 =
\frac{\pi(\rho)}{\mathrm{deg}(\rho)} \hat P_\omega(X_{-m}=x_{-m},\dots, X_m=x_m).
\end{align*}
\normalsize
Analogously, we obtain the same convergence result, if there exists a $x_{-(k-1)}$ along the path in negative time with $\sum_{v\sim x_{-(k-1)}} \xi(x_{-(k-1)},v) = 0$, which implies the continuity of the last two factors integrated in \eqref{finitetreefinitesum}. 

Moreover, the first $m$ generations of the trees $T^{(n)}$ and $T$, rooted at $\rho^{(n)}$ and $\rho$, respectively, are equal for $n$ large enough. This implies $\big\{\big|T^{(n)}_{|B_m(\rho^{(n)})}\big|\le M\big\}=\{|T_{| B_m(\rho)}|\le M\}$ for $n$ big enough. Thus, the function $\mathds 1_{\{|T\vert_{B_m(\rho)}|\le M\}}$ is continuous in $\omega=(T,\rho,\xi)$, 
which completes the proof of the continuity of the functions integrated in \eqref{finitetreefinitesum}.   

Using that also $\gamma_\varepsilon\to \gamma_0$, we obtain that
\small
\begin{align*}
&\lim_{\varepsilon\to 0} \int f((x_n)_{n\in \mathbb{Z}},\omega) \mathds 1_{\{|T\vert_{B_m(\rho)}|\le M\}} \,\D \PPehat
\\&= \sum_{x_{-m},\dots,x_m} \int g((x_{-m},\dots,x_m),\omega) \mathds 1_{\{T\vert_{B_m(\rho)}\le M\}} \frac{\pi(\rho)}{\gamma_0\mathrm{deg}(\rho)} \hat P_\omega(X_{-m}=x_{-m},\dots, X_m=x_m) \,\D \mathrm{P}^\mathrm{aug}_0
\\&=\int f((x_n)_{n\in \mathbb{Z}},\omega) \mathds 1_{\{|T\vert_{B_m(\rho)}|\le M\}} \,\D \hat{\mathbb{P}}_0 .
\end{align*}
\normalsize
This computation shows that the first summand in \eqref{1} converges to zero for $\varepsilon\to 0$. Let us now consider the second one. We have
\begin{align*}
\left|\int f \mathds 1_{\{|T_{| B_m(\rho)}|> M\}} \,\D \PPehat \right| 
\le \Vert f\Vert_{\infty} \mathrm{P}^\mathrm{aug}_\varepsilon(|T_{|B_m(\rho)}|> M).
\end{align*}
The last probability does not depend on $\varepsilon$, so according to the assumption of the locally finitness of $T$ we can bound this term by $\delta(M)$ independent of $\varepsilon$ with $\delta(M)\to 0$ for $M\to\infty$. For the third summand of \eqref{1} we obtain an analogous upper bound, which yields the weak convergence of $\PPehat$ on $\F_m\otimes\G$.\\

To show the weak convergence on $\F\otimes\G$ we follow the arguments of \cite[Theorem 2.2]{billingsley2013convergence}. To apply the Portmanteau Theorem, note that $\F\otimes\G$ is the Borel-$\sigma$-algebra on $\mathbb{T}^\mathbb{Z}\times\Omega$ and $\bigcup_m\mathcal{F}_m$ contains a basis of the topology on $\mathbb{T}^\mathbb{Z}$. 

Let $G=G_1\times G_2\in\F\otimes\G$ be an open set. Hence, $G_1$ is open and we can write it as $G_1=\bigcup_{i=1}^\infty A_i$ for some $A_i\in\bigcup_m\F_m$. By continuity of probability measures we have
\begin{align*}
\hat{\mathbb{P}}_0 \left(\bigcup_{i=1}^N A_i\times G_2\right) \xrightarrow[N\to\infty]{} \hat{\mathbb{P}}_0 \left(G_1\times G_2\right) .
\end{align*}
Given an arbitrary $\delta>0$ we can choose $r\in \N$, such that
\begin{align}
\hat{\mathbb{P}}_0 \left(\bigcup_{i=1}^r A_i\times G_2\right) >\hat{\mathbb{P}}_0 \left(G_1\times G_2\right) -\delta. \label{2}
\end{align}
We obtain
\begin{align*}
\liminf_{\varepsilon\to 0} \PPehat(G_1\times G_2)
&\ge \liminf_{\varepsilon\to 0} \PPehat \left(\bigcup_{i=1}^r A_i\times G_2\right)
\\&\ge \hat{\mathbb{P}}_0 \left(\bigcup_{i=1}^r A_i\times G_2\right)
\\& > \hat{\mathbb{P}}_0 \left(G_1\times G_2\right) -\delta ,
\end{align*}
where the second inequality follows from the weak convergence of $\PPehat$ on $\F_m\otimes\G$ and the Portmanteau-Theorem and \eqref{2} was used for the last inequality. Since $\delta>0$ was arbitrary we have 
\begin{align}
\liminf_{\varepsilon\to 0} \PPehat (G)\geq \hat{\mathbb{P}}_0 \left( G\right)
\end{align}
for all open sets $G\in\F\otimes\G$, that is, $\PPehat$ converges weakly to $\hat{\mathbb{P}}_0$. 
\hfill $\Box$

\bibliographystyle{alpha}
\bibliography{bibGWT}

\bigskip

{\footnotesize
\noindent
TU Dortmund, \\
Fakult\"at f\"ur Mathematik, \\
Vogelpothsweg 87, 
44227 Dortmund, 
Germany, \\
tabea.glatzel@tu-dortmund.de\\
jan.nagel@tu-dortmund.de
}

\end{document}